\documentclass[reqno, 11pt]{amsart} 

\usepackage{amssymb}
\usepackage{amsfonts}
\usepackage[latin1]{inputenc} 
\usepackage{amsmath, amscd}
\usepackage{graphicx}
\newtheorem{theorem}{Theorem}[section]
\newtheorem{proposition}[theorem]{Proposition}
\newtheorem{lemma}[theorem]{Lemma}
\newtheorem{corollary}[theorem]{Corollary}

\newtheorem{remark}{Remark}[section]

\newcommand{\no}{\noindent}
\newcommand{\nn}{\nonumber}
\newcommand{\oo}{\infty}

\newcommand{\ra}{\rightarrow}
\newcommand{\comp}{\circ}

\newcommand{\N}{\mathbb{N}}

\newcommand{\C}{\mathbb{C}}
\newcommand{\R}{\mathbb{R}}
\makeatletter
\@addtoreset{equation}{section}
\makeatother

\renewcommand\thetable{\thesection.\@arabic\c@table}

\bibdata{prob}
\bibstyle{alpha}

\title[Spectrum for operators of stochastic machines]{Spectrum of stochastic adding machines and fibered Julia sets}

\author{Ali Messaoudi$^1$, Olivier Sester$^2$, Glauco Valle$^3$}
\thanks{1. Supported by CNPq grant 305939/2009-2 and Fapesp Project 2011/23199-1.}
\thanks{2. Supported by  Fapesp Project	2007/06896-5.}
\thanks{3. Supported by CNPq grant, 307038/2009-2.}

\date{\today}

\address{
\newline
Ali Messaoudi
\newline
UNESP - Departamento de matem\'atica do Instituto de Bioci\^encias Letras e Ci\^encias Exatas de S\~ao Jos\'e do Rio Preto
\newline
e-mail: {\rm \texttt{messaoud@ibilce.unesp.br}}
\newline
\newline
Olivier Sester
\newline
LAMA - Universit\'e Paris-Est Marne-la-Vall\'ee, UMR CNRS 8050, France
\newline
e-mail: {\rm \texttt{olivier.sester@univ-mlv.fr}}
\newline
\newline
Glauco Valle
\newline
UFRJ - Departamento de m\'etodos estat\'{\i}sticos do Instituto de Matem\'atica.
\newline  Caixa Postal 68530, 21945-970, Rio de Janeiro, Brasil
\newline
e-mail: {\rm \texttt{glauco.valle@im.ufrj.br}}
}

\subjclass[2000]{primary 37A30, 37F50; secondary 47A10}
\keywords{Julia sets, Stochastic adding machines, Markov chains, Spectrum of transition operator}

\begin{document}

\maketitle

\begin{abstract}
Consider the basic algorithm to perform the transformation $n\mapsto n+1$ changing digits of the $d$-adic expansion of $n$ one by one. We obtain a family of Markov chains on the non-negative integers through sucessive and independent applications of the algorithm modified by a parametrized stochastic rule that randomly prevents one of the steps in the algorithm to finish. The objects of study in this paper are the spectra of the transition operators of these Markov chains. The spectra of these Markov chains turn out to be fibered Julia sets of fibered polynomials. This enable us to analyze their topological and analytical properties with respect to the underlying
parameters of the Markov chains.
\end{abstract}

\section{Introduction}
\label{sec:intro}

Binary representations of real numbers have many useful applications in science. One cares not only on how transformations on sets of real numbers can be described through their binary representations, but also on how these transformations can be performed algorithmically. The transformation that associate to a natural number its successor, adding one to the number, is one of the simplest to be described by binary representations. A basic algorithm to perform the transformation $n \mapsto n+1$, changing binary digits one by one, requires less than $\left\lfloor \log_2(n) \right\rfloor + 1$ steps. Killeen and Taylor \cite{kt} proposed a stochastic rule that randomly prevents one of the steps in the algorithm to finish resulting in a number smaller than $n+1$. Successive iterations of the Killeen and Taylor rule give rise to a Markov chain on $\mathbb{Z}_+ = \{0,1,2,3,...\}$ whose transition operator has important spectral properties. In particular, the spectrum of the transition operator in $l^{\infty}$ is equal to the filled-in Julia set of a quadratic map. Here we propose a generalization of the Killeen and Taylor machine and also of the results obtained in \cite{am} and \cite{kt}.

We extend these results in two directions. On the one hand, we consider not only binary representations but also $d$-adic expansions of the natural numbers. On the other hand, the stochastic rule we consider
is more general: the iteration of the adding algorithm is randomized through Bernoulli variables whose parameters change at each step (see below).

\medskip
Let us fix a positive integer $d \ge 2$.  Set
$$
\Gamma = \Gamma_d := \Big\{ (a_j)_{j=1}^{+\oo} \in \{0,...,d-1\}^\mathbb{N} : \sum_{j=1}^{+\oo} a_j < \infty \Big\} \, .
$$
There is a one to one map from $\mathbb{Z}_+$ to $\Gamma$ that associates to each $n$ a sequence $(a_{j}(n))_{j=1}^{+\oo}$ such that
$$
n = \sum_{j=1}^{+\oo} a_{j}(n) d^{j-1} \, .
$$
The right hand side of the previous equality is called the $d$-adic expansion of $n$ and $a_{j}(n)$ is called the $j$th digit of the expansion. The map $n \mapsto n+1$ operates on $\Gamma$ in the following way: we define the counter $\zeta_n = \zeta_{d,n} := \min\{j\ge 1:a_{j}(n) \neq d-1\} $ then
$$
a_{j}(n+1) = \left\{
\begin{array}{cl}
0 &, \ j< \zeta_{n} \\
a_{j}(n) + 1 &, \ j = \zeta_{n} \\
a_{j}(n) &, \ j>\zeta_{n} \, .
\end{array}
\right.
$$
So an adding machine algorithm, that maps $n$ to $n+1$ using $d$-adic expansions by changing one digit on each step, is performed in $\zeta_{n}$ steps in the following way: the first $\zeta_{n} -1$ digits are replaced by zero recursively and in $\zeta_{n}$th step we add one to the $\zeta_{n}$th digit (basically we are adding one modulus $d$ on each step). Note that $0\le \zeta_{n} \le \left\lfloor \log_d(n) \right\rfloor + 1$.

As an example consider $d=3$ and $n= 98 = 2 \cdot 3^0 + 2 \cdot 3^1 + 1 \cdot 3^2 + 1 \cdot 3^4$, then the adding machine algorithm is performed in $\zeta_{3,98} = 3$ steps as follows:
\begin{equation}
\label{AM98}
22101 \mapsto 02101 \mapsto 00101 \mapsto 00201 \, .
\end{equation}

\medskip

Now suppose that for each step of the adding machine algorithm, independently of any other step, there is a positive probability that the information about the counter is lost, thus making the algorithm  stop. This implies that the outcome of the adding machine is a random variable. We call this procedure the adding machine algorithm with fallible counter, or simply $\rm{AMFC}_d$ where $d$ represents the base.

Formally, we fix a sequence $(p_j)_{j=1}^{+\oo}$ of real numbers in $(0,1]$ and a sequence $(\xi_j)_{j=1}^{+\oo}$ of independent random variables such that $\xi_j$ is a Bernoulli distribution with parameter $p_j$. Define the random time $\tau = \inf\{ j : \xi_j = 0 \}$. Then the $\rm{AMFC}_d$ is defined by applying the adding machine algorithm to $n$ and stopping at the step $\tau$ if $\tau \le \zeta_{n}$ (this means that steps $j\ge \tau$ are not performed).

Let us return to the previous example with $d=3$ and $n=98$. If $\xi_1 = 1$ and $\xi_2=0$, which occurs with probability $p_1(1-p_2)$ then $\tau =2 < 3 = \zeta_{98}$. Thus the $\rm{AMFC}_3$ applied to $n=98$ stops just after step one giving an outcome of $2 \cdot 3 + 1 \cdot 3^2 + 1 \cdot 3^4 = 96$, see scheme in \ref{AM98}. Indeed the probability distribution of the outcome of the $\rm{AMFC}_3$ applied to $n=98$ is
$$
22101 \, (98) \mapsto \left\{
\begin{array}{cl}
22101 \, (98) \, , & \textrm{with probability } 1-p_1 \\
02101 \, (96) \, , & \textrm{with probability } p_1 (1-p_2) \\
00101 \, (90) \, , & \textrm{with probability } p_1 p_2 (1-p_3) \\
00201 \, (99) \, , & \textrm{with probability } p_1 p_2 p_3 \, .
\end{array}
\right.
$$

\smallskip

Now fix an initial, possibly random, state $X(0) \in \mathbb{Z}_+$. We apply recursively the $\rm{AMFC}_d$ to its successive outcomes starting at $X(0)$ and using independent sequences of Bernoulli random variables at different times. These random sequences are associated to the same fixed sequence of probabilities $(p_j)_{j=1}^{+\oo}$. In this way, we generate a discrete time-homogeneous Markov chain $(X(t))_{t \ge 0}$. This Markov chain is irreducible if and only if $p_j < 1$ for infinitely many j's. During the rest of the paper we will assume that the previous condition is satisfied and the chain is irreducible. For a concise introduction to the Theory of discrete time Markov chains with countable state spaces, we suggest Chapter 2 in \cite{l}. For a more specific book on Markov Chains, we suggest \cite{R}.

As mentioned before, Killeen and Taylor \cite{kt} considered  this stochastic machine in the case $d=2$ and $p_n= p$ for all $n \in \mathbb{N}$.
They proved that the spectrum of the transition operator in $l^{\infty}$ is equal to the filled-in Julia set of a quadratic map.
In \cite{am}, El Abdalaoui and Messaoudi, also in the case $d=2$ and $p_n$ constant, studied the spectrum of the transition operator acting in other Banach spaces as $c_0$ and $l^{\alpha}(\mathbb{Z}_+),\; \alpha \geq 1$.
Messaoudi and Smania \cite{ms} also defined the stochastic adding machine in the case where the base of numeration is not constant.
In particular, they considered the case where the base is the Fibonacci sequence. They proved that the eigenvalues of the spectrum of the transition operator acting in $l^{\infty}(\mathbb{Z}_+)$ is connected to the Julia set of an endomorphism of $\mathbb{C}^2$ (see also \cite{U} for the case where the base belongs to a class of recurrent sequences of degree $2$).

In this paper, we  study  convergence and spectral properties of the $\rm{AMFC}_d$  Markov chain. We prove that the $\rm{AMFC}_d$ Markov chain is null recurrent if and only if $\prod_{j=1}^{+\oo} p_j = 0 $. Otherwise the chain is transient.
We also prove that the spectrum of its transition operator acting on $l^{\infty}$ is equal to the filled-in fibered Julia set $E$ defined by
$$
E = E_d := \Big\{ z \in \mathbb{C} : \limsup_{j\ra +\oo} |\tilde{f}_{j}(z)| < +\oo \Big\},
$$
where
$\tilde{f}_{j} := f_{j} \comp ... \comp f_{1}$ for all $j \geq 1$ and
 $f_{j}:  \mathbb{C} \ra \mathbb{C}$, is the function defined by
$$
f_{j}(z) := \left( \frac{z-(1-p_j)}{p_j} \right)^d \, .
$$
We shall study the topological properties of the filled-in fibered Julia set $E$. In particular, we give sufficient conditions on the sequence $(p_n)_{n \geq 1}$ to ensure that $E$ is a connected set, or has a finite number of connected components. 

We also study some properties of the fibered Julia set $\partial {E}$, in particular, we introduce the Green function of $\partial {E}$  and prove that there exist $0<\rho<1$ and $\kappa>1$ such that whenever $p_i\in[\rho,1] $ for all $i\geq 2$, then $\partial {E}$ is a $\kappa$-quasicircle.

The paper is organized as follows: In section \ref{sec:chain} we obtain the transition operator of the $\rm{AMFC}_d$ Markov chain and we give a necessary and sufficient condition for recurrence and transience; Section \ref{sec:spectra} is devoted to provide an exact description of the spectra of these transition operators acting on $l^{\infty}$; Section \ref{sec:connectedness}
contains results about connectedness properties of the filled-in fibered Julia sets $E$; In section \ref{sec:fibered}, further properties of $\partial {E}$ are established in connection with properties of the associated fibered polynomials.


\section{Transition operators and recurrence of $\rm{AMFC}_d$ chains}
\label{sec:chain}

\smallskip

In this section $(X(t))_{t \ge 0}$ is an irreducible $\rm{AMFC}_d$ Markov chain associated to a sequence of probabilities $\bar{p} = (p_j)_{j=1}^{+\oo}$. Our first aim is to describe the transition probabilities of $(X(t))_{t \ge 0}$ which we denote $s(n,m) = s_{\bar{p},d}(n,m):=P(X(t+1)=m|X(t)=n)$.
They can be obtained directly from the description of the chain. First recall the definition of the counter $\zeta_n$ from Section \ref{sec:intro}. For every $n\ge 0$, we have
\begin{equation}
\label{tp1}
s(n,m) =
\left\{
\begin{array}{cl}
(1-p_{r+1}) \prod_{j=1}^r p_j &\rm{if} \ m = n - \sum_{j=1}^r (d-1) d^{j-1} \, , \\
& \  \ \ r \le \zeta_n -1 \, , \ \zeta_n \ge 2 \, , \\
1-p_1 &\rm{if}  \ m=n \, , \\
\prod_{j=1}^{\zeta_n} p_j &\rm{if} \ m=n+1 \, , \\
0 & \ \textrm{otherwise} \, .
\end{array}
\right.
\end{equation}
From the exact expressions above, the transition probabilities satisfy a property of self-similarity. Indeed, the following Lemma is straightforward
\begin{lemma} For all $j\geq 2$ and for all $d^{j-1} \le n \le d^{j}-2$, we have that $\zeta_n \le j$ and
$$ s(n,m) =
\left\{
\begin{array}{cl}
s(n - a_{j}(n) d^{j-1}, m - a_{j}(n) d^{j-1})  &, \ d^{j-1} \le m \le d^{j}-1 \, , \\
0 &, \ \textrm{otherwise.}
\end{array}
\right.
$$
Moreover, if $n=d^j-1$, we have $\zeta_n = j+1$ thus $s(d^j -1, d^j) = \prod_{l=1}^{j+1} p_l$ and
\begin{equation}
\label{tp4}
s(d^j -1, d^j - d^r) = (1-p_{r+1}) \prod_{l=1}^{r} p_l \, , \ 1\le r \le j \, .
\end{equation}
\end{lemma}

With the transition probabilities, we obtain the countable transition matrix of the $\rm{AMFC}_d$ Markov chain $S = S_{d} = [s(n,m)]_{n,m \ge 0}$. To help the reader, the first entries of the matrix $S_2$ are given below:
$$
\tiny{
\left[
\begin{array}{cccccccccc}
\!\!1-p_1           \!\!&\!\! p_1 \!\!&\!\!0         \!\!&\!\!0    \!\!&\!\!0            \!\!&\!\!0    \!\!&\!\!0         \!\!&\!\!0    \!\!&\!\!0           \!\!&\!\! \cdots \!\! \\
\!\!p_1(1-p_2)      \!\!&\!\!1-p_1\!\!&\!\!p_1p_2    \!\!&\!\!0    \!\!&\!\!0            \!\!&\!\!0    \!\!&\!\!0         \!\!&\!\!0    \!\!&\!\!0           \!\!&\!\! \cdots \!\! \\
\!\!0               \!\!&\!\!0    \!\!&\!\!1-p_1     \!\!&\!\!p_1  \!\!&\!\!0            \!\!&\!\!0    \!\!&\!\!0         \!\!&\!\!0    \!\!&\!\!0           \!\!&\!\! \cdots \!\! \\
\!\!p_1p_2(1-p_3)   \!\!&\!\!0    \!\!&\!\!p_1(1-p_2)\!\!&\!\!1-p_1\!\!&\!\!p_1p_2p_3    \!\!&\!\!0    \!\!&\!\!0         \!\!&\!\!0    \!\!&\!\!0           \!\!&\!\! \cdots \!\! \\
\!\!0               \!\!&\!\!0    \!\!&\!\!0         \!\!&\!\!0    \!\!&\!\!1-p_1        \!\!&\!\! p_1 \!\!&\!\!0         \!\!&\!\!0    \!\!&\!\!0           \!\!&\!\! \cdots \!\! \\
\!\!0               \!\!&\!\!0    \!\!&\!\!0         \!\!&\!\!0    \!\!&\!\!p_1(1-p_2)   \!\!&\!\!1-p_1\!\!&\!\!p_1p_2    \!\!&\!\!0    \!\!&\!\!0           \!\!&\!\! \cdots \!\! \\
\!\!0               \!\!&\!\!0    \!\!&\!\!0         \!\!&\!\!0    \!\!&\!\!0            \!\!&\!\!0    \!\!&\!\!1-p_1     \!\!&\!\!p_1  \!\!&\!\!0           \!\!&\!\! \cdots \!\! \\
\!\!p_1p_2p_3(1-p_4)\!\!&\!\!0    \!\!&\!\!0         \!\!&\!\!0    \!\!&\!\!p_1p_2(1-p_3)\!\!&\!\!0    \!\!&\!\!p_1(1-p_2)\!\!&\!\!1-p_1\!\!&\!\!p_1p_2p_3p_4\!\!&\!\! \cdots \!\! \\
\!\!\vdots          \!\!&\!\!\vdots \!\!&\!\!\vdots    \!\!&\!\!\vdots \!\!&\!\!\vdots       \!\!&\!\!\vdots \!\!&\!\!\vdots    \!\!&\!\!\vdots \!\!&\!\!\vdots      \!\!&\!\!\ddots
\end{array}
\right]}
$$
For $S_3$, the first entries of the matrix are given below:
$$
\tiny{
\left[
\begin{array}{ccccccccccc}
\!\!1-p_1        \!\!&\!\!p_1  \!\!&\!\!0    \!\!&\!\!0         \!\!&\!\!0    \!\!&\!\!0    \!\!&\!\!0         \!\!&\!\!0    \!\!&\!\!0    \!\!&\!\!0  \!\!&\!\! \cdots \!\! \\
\!\!0            \!\!&\!\!1-p_1\!\!&\!\!p_1  \!\!&\!\!0         \!\!&\!\!0    \!\!&\!\!0    \!\!&\!\!0         \!\!&\!\!0    \!\!&\!\!0    \!\!&\!\!0  \!\!&\!\! \cdots \!\! \\
\!\!p_1(1-p_2)   \!\!&\!\!0    \!\!&\!\!1-p_1\!\!&\!\!p_1p_2    \!\!&\!\!0    \!\!&\!\!0    \!\!&\!\!0         \!\!&\!\!0    \!\!&\!\!0    \!\!&\!\!0  \!\!&\!\!\cdots \!\! \\
\!\!0            \!\!&\!\!0    \!\!&\!\!0    \!\!&\!\!1-p_1     \!\!&\!\!p_1  \!\!&\!\!0    \!\!&\!\!0         \!\!&\!\!0    \!\!&\!\!0    \!\!&\!\!0  \!\!&\!\! \cdots \!\! \\
\!\!0            \!\!&\!\!0    \!\!&\!\!0    \!\!&\!\!0         \!\!&\!\!1-p_1\!\!&\!\!p_1  \!\!&\!\!0         \!\!&\!\!0    \!\!&\!\!0    \!\!&\!\!0  \!\!&\!\! \cdots \!\! \\
\!\!0            \!\!&\!\!0    \!\!&\!\!0    \!\!&\!\!p_1(1-p_2)\!\!&\!\!0    \!\!&\!\!1-p_1\!\!&\!\!p_1p_2    \!\!&\!\!0    \!\!&\!\!0    \!\!&\!\!0  \!\!&\!\! \cdots \!\! \\
\!\!0            \!\!&\!\!0    \!\!&\!\!0    \!\!&\!\!0         \!\!&\!\!0    \!\!&\!\!0    \!\!&\!\!1-p_1     \!\!&\!\!p_1  \!\!&\!\!0    \!\!&\!\!0  \!\!&\!\! \cdots \!\! \\
\!\!0            \!\!&\!\!0    \!\!&\!\!0    \!\!&\!\!0         \!\!&\!\!0    \!\!&\!\!0    \!\!&\!\!0         \!\!&\!\!1-p_1\!\!&\!\!p_1  \!\!&\!\!0  \!\!&\!\! \cdots \!\! \\
\!\!p_1p_2(1-p_3)\!\!&\!\!0    \!\!&\!\!0    \!\!&\!\!0         \!\!&\!\!0    \!\!&\!\!0    \!\!&\!\!p_1(1-p_2)\!\!&\!\!0    \!\!&\!\!1-p_1\!\!&\!\!p_1p_{2}p_{3}\!\!&\!\! \cdots \!\! \\
\!\!\vdots \!\!&\!\!\vdots \!\!&\!\!\vdots \!\!&\!\!\vdots \!\!&\!\!\vdots \!\!&\!\!\vdots \!\!&\!\!\vdots \!\!&\!\!\vdots \!\!&\!\!\vdots \!\!&\!\!\vdots \!\!&\!\! \ddots \!\!
\end{array}
\right]}
$$
The transition operator induced by $S$, acting on $l^\oo(\mathbb{Z}_+)$, will also be denoted by $S$. From a result in \cite{am} we see that its restriction on $l^\alpha(\mathbb{Z}_+)$, $\alpha \ge 1$, is a well-defined operator on each of these spaces. Note that $S$ is doubly stochastic, i.e, $S$ is stochastic and the sum of the entries in each column is equal to one, if and only if $\prod_{j=1}^{+\oo} p_j = 0$. Indeed, except for the first column whose sum is $1- \prod_{j=1}^{+\oo} p_j$, any other row and column of $S$ has the sum of its entries equal to one.
\medskip

In the next Proposition, we obtain a necessary and sufficient condition for recurrence of the $\rm{AMFC}_d$ Markov chain.
\begin{proposition}
\label{prop:recurrence}
The $\rm{AMFC}_d$ Markov chain is null recurrent if and only if
\begin{equation}
\label{reccondition}
\prod_{j=1}^{+\oo} p_j = 0 \, .
\end{equation}
Otherwise the chain is transient.
\end{proposition}

\begin{proof}
We start by showing that condition (\ref{reccondition}) is necessary and sufficient to guarantee the recurrence of the $\rm{AMFC}_d$ Markov chain. The $\rm{AMFC}_d$ Markov chain is transient if and only if there exists a sequence $v=(v_j)_{j=1}^{+\oo}$ such that $0< v_j \le 1$ and
\begin{equation}
\label{rectran1}
v_j = \sum_{m = 1}^{+\infty} s(j,m) \, v_m \, , \ j \ge 1 \, ,
\end{equation}
i.e, $\tilde{S}_d \, v = v$ where $\tilde{S}_d$ is obtained from $S$ by removing its first line and column. Indeed, in the transient case a solution is obtained by taking $v_m$ as the probability that $0$ is never visited by the $\rm{AMFC}_d$ Markov chain given that the chain starts at state $m$, see the discussion on pages 42-43 Chapter 2 in~\cite{l}  and also Corollary~16.48 in~\cite{O}. 

\smallskip
Suppose that $v=(v_j)_{j=1}^{+\oo}$ satisfies the above conditions. We claim that
\begin{equation}
\label{rectran2}
v_{d^l + j} = v_{d^l} \, , \, \textrm{ for every } \, l\ge 0 \, \textrm{ and } \, j \in \{ 1, ..., (d-1)d^l-1 \} \, .
\end{equation}
The proof follows by induction. Indeed, for $j \in \{ 1, ..., (d-1)d^l-1 \}$, suppose that $v_{d^l} = v_{d^l + r}$, for all $0 \le r \le j-1$  then
$v_{d^l+j-1}  =  \sum_{m=1}^j s(d^l+j-1,m) \, v_{m}$. Since
$s(d^l+j-1,m)=0$  for all $0 \leq m < d^l$, we have
\begin{eqnarray}
\label{rectran3}
v_{d^{l}+j-1} & = & \sum_{r=0}^j s(d^l+j-1,d^l+r) \, v_{d^l + r} \nn \\
& = & \left( \sum_{r=0}^j s(d^l+j-1,d^l+r) \right) \, v_{d^l} \nn \\
& & \quad \quad + s(d^l+j-1,d^l+j) (v_{d^l+j} - v_{d^l})  \, .
\end{eqnarray}
Using the fact that $j \in \{ 1, ..., (d-1)d^l-1 \}$ note that
$$
\sum_{r=0}^j s(d^l+j-1,d^l+r) = 1 \, ,
$$
thus, since $s(d^l+j-1,d^l+j)>0$, from (\ref{rectran3}), we have that $v_{d^l+j} = v_{d^l+j-1} = v_{d^l}$. This proves the claim.

\smallskip

It remains to obtain $v_{d^{l+1}}$ from $v_{d^{l}}$ for $l\ge 0$. First note that (\ref{rectran2}) implies $v_{d^{l+1} - d^r} = v_{d^l}$ for $0\le r \le l$. From the transition probabilities expression in (\ref{tp4}),
if we let $p_0=1$, we have that
\begin{eqnarray*}
v_{d^l} & = & v_{d^{l+1}-1} \ = \ (p_0...p_{l+2}) v_{d^{l+1}} + \sum_{r=0}^{l} (p_0...p_r - p_0...p_{r+1}) v_{d^{l+1} - d^{r}} \\
& = & (p_0...p_{l+2}) v_{d^{l+1}} + (1 - p_0...p_{l+1}) v_{d^{l}}.
\end{eqnarray*}
Therefore for every $l\ge 1$,
$$
v_{d^{l}} = \frac{v_{d^{l-1}}}{p_{l+1}} = \frac{v_1}{\prod_{j=2}^{l+1} p_j} \, .
$$
From this equality, we get to the conclusion that $v$ exists and the chain is transient if and only if
$$
\prod_{j=1}^{+\oo} p_j > 0 \, .
$$

\smallskip

Now suppose that we are in the recurrent case. Since $S$ is an irreducible countable doubly stochastic matrix, it is simple to verify that the $\rm{AMFC}_d$ has no invariant probability measure and so cannot be positive recurrent. Indeed, since each column of $S$ has the sum of its entries equal to one, only the constant sequences are right eigenvectors of $S$ with eigenvalue $1$. But, by definition, an invariant probability measure is an $l^1(\mathbb{Z}_+)$ sequence which is also a right eigenvector of $S$ with eigenvalue $1$. Thus $S$ has no invariant probability measure. 
\end{proof}

We can also obtain the recurrence/transience condition of Proposition \ref{prop:recurrence} through probabilistic arguments. We just describe roughly these arguments leaving the details to the reader. Let $(X(t))_{t \ge 0}$ be the $\rm{AMFC}_d$ Markov chain starting at $X(0)=0$. For $n\ge 0$, denote the first hitting time of state $d^n$ by
$$
\tau_n = \min \{ t \ge 1 : X(t) = d^n \}
$$
and the number of visits to $0$ before the random time $\tau_n$ by $N_n$. Then the expectation of $N_n$ is
\begin{equation}
\label{eq:taun}
\Big( \prod_{j=1}^{n+1} p_j \Big)^{-1} \, .
\end{equation}
We check this by induction. If $n=0$ then $N_0 = \tau_0$. But $\tau_0$ is a geometric random variable with parameter $p_1$, which has expectation $p_1^{-1}$. Now suppose that (\ref{eq:taun}) holds for $n$. When the chain reaches state $d^n$, it must get to $d^{n+1} - 1$ to attempt a return to $0$. If it returns to $0$, it must spend a time with the same distribution of $\tau_n$ to reach $d^n$ again. Moreover, the number of visits to $d^{n+1} - 1$ resulting in a jump to $0$ or a jump to $d^{n+1}$ is a geometric random variable of parameter $p_{n+2}$. By the Markov property, the expectation of $N_{n+1}$ is the expectation of $N_n$ times $p_{n+2}^{-1}$. This means that (\ref{eq:taun}) also holds for $n+1$.

Letting $n$ go to infinity, we get that the expected number of visits to $0$ is $\Big( \prod_{j=1}^{\infty} p_j \Big)^{-1}$. Moreover, as it is shown in the beggining of page 40 in \cite{l}, the probability that the chain never returns to $0$ is the inverse of the expected number of visits to $0$, which is $\prod_{j=1}^{\infty} p_j$. Therefore, the chain is recurrent if and only if $\prod_{j=1}^{\infty} p_j = 0$.

Finally, let us just point out that the previous argument also allows us to show that the expectation of $\tau_n$ is given by $d^n \Big( \prod_{j=1}^{n+1} p_j \Big)^{-1}$. Indeed, the expected number of visits to each one of the states $0,...,d^n-1$ before time $\tau_n$ is $\Big( \prod_{j=1}^{n+1} p_j \Big)^{-1}$.

\medskip

\section{Spectra of transition operators of $\rm{AMFC}_d$ chains}
\label{sec:spectra}
%
%
In this section we describe the spectrum of the transition operator of an $\rm{AMFC}_d$ acting on $l^\infty (\mathbb{Z}_+)$ for a fixed sequence  $\bar{p} = (p_j)_{j=1}^{+\oo}$ with $p_{j}$ in $(0,1]$ satisfying the irreducibility condition. We start introducing some notation. 
Let $f_{j}:  \mathbb{C} \ra \mathbb{C}$, $j \ge 1$, be the function defined by
$$
f_{j}(z) := \left( \frac{z-(1-p_j)}{p_j} \right)^d \, .
$$
Also let $\tilde{f}_0 := Id$, $\tilde{f}_{j} := f_{j} \comp ... \comp f_{1}$ and
$$
E_{\bar{p}}= E := \Big\{ z \in \mathbb{C} : \limsup_{j\ra +\oo} |\tilde{f}_{j}(z)| < +\oo \Big\} \, .
$$
We use the notation $\mathbb{D}(w,r) = \{ z\in \mathbb{C} : |w-z| < r \}$ and $\overline{\mathbb{D}(w,r)} =  \{ z\in \mathbb{C} : |w-z| \leq r \}$.
\begin{lemma}\label{bound}
The set $E_{\bar{p}}$ is included in the closed disk $\overline{\mathbb{D}(1-p_1,p_1)}$. Moreover, for all $z\in E_{\bar{p}}$ and $j\geq 1$, $\tilde{f}_{j}(z)$ belongs to the disk $\overline{\mathbb{D} (1-p_{j+1},p_{j+1})}$.
\end{lemma}
\begin{proof}
Take $p_j\in (0,1)$ and $z\in \mathbb{C}$ with $|z|>1$ then
\begin{equation}
\label{LowerBound}
\left| \frac{z-(1-p_j)}{p_j}\right| \ge \frac{|z| - (1-p_j)}{p_j} = \frac{|z|-1}{p_j} + 1 > |z| > 1 \, .
\end{equation}
Thus, we obtain, for every $z\in \mathbb{C}$ with $|z|>1$ and $j\ge 1$, that
$$
|f_{j}(z)| > \vert z \vert^d >1 \, .
$$

Now suppose $|\tilde{f}_{r}(z)|>1$ for some $r>1$, then by induction one can show that for $j>r$
\begin{equation}
\label{julia1}
|\tilde{f}_{j}(z)| \ge |\tilde{f}_{r}(z)|^{d^{j-r}} \, .
\end{equation}

From (\ref{julia1}) we see that $\lim_{j\ra +\oo} |\tilde{f}_{j}(z)| = +\oo$ whenever $|\tilde{f}_{r}(z)|>1$ for some $r>1$.
In particular, if $|\tilde{f}_{r}(z)|>1$ then $z \notin E$.

Now we can finish the proof. Suppose that $|z-(1-p_1)|>p_1.$ This implies that $|f_1(z)|>1$ and so $z\notin E$.
Analogously, if $|\tilde{f}_j(z) - (1-p_{j+1})| > p_{j+1}$, we have that $|\tilde{f}_{j+1}(z)|>1$ and then $z \notin E$.
\end{proof}

\smallskip
\begin{corollary}
\label{cor:bound}
We have the following equality
\begin{eqnarray*}
E_{\bar{p}} & = & \overline{\mathbb{D}(1-p_1,p_1)} \cap \bigcap_{j=1}^\infty \tilde{f}_j^{-1} \big( \overline{\mathbb{D}(1-p_{j+1},p_{j+1})} \big) \\
& = &
\overline{\mathbb{D}(0,1)} \cap \bigcap_{j=1}^\infty \tilde{f}_j^{-1} \big( \overline{\mathbb{D} (0,1)} \big) \, .
\end{eqnarray*}
\end{corollary}

\smallskip
\begin{proposition}
The point spectrum of $S$ in $l^\oo (\mathbb{\mathbb{Z}_+})$ is equal to $E$. Furthermore, fix $\lambda \in E$ and $v_0 >0$ and define
\begin{equation}
\label{ps1}
v_n = v_0 \, \prod_{r=1}^{\left\lfloor \log_d(n) \right\rfloor + 1} (q_{\lambda}(r))^{a_{r}(n)} \, , \ n \ge 1 \, ,
\end{equation}
where $a_{r}(n)$ is the $r$th digit of $n$ in its $d$-adic expansion and
\begin{equation}
\label{ps2}
q_{\lambda}(r) = (\, h_r \comp \tilde{f}_{r-1} \, ) (\lambda)
\end{equation}
with
\begin{equation}
\label{ps3}
h_r(z) = \frac{z}{p_r} - \frac{1-p_r}{p_r}.
\end{equation}
Thus $(v_n)_{n=0}^{+\oo}$ is, up to multiplication by a constant, the unique right eigenvector of $S$ in $l^\oo (\mathbb{N})$ with eigenvalue $\lambda$.
\end{proposition}

\begin{proof}
Since $S$ is stochastic, its spectrum is a subset of the closed disk $\overline{\mathbb{D}(0,1)}$. Let us fix $\lambda \in \overline{\mathbb{D}(0,1)}$ and let $(v_n)_{n=0}^{+\oo}$ be the sequence defined in the statement. Then the proof of the Proposition relies on the two following claims below:

\medskip
%

\no \textbf{Claim 1:} $(|q_{\lambda}(j)|)_{j=1}^{+\oo}$ is bounded above by one if and only if $\lambda \in E$, otherwise it is unbounded. In particular,  $(v_n)_{n=0}^{+\oo}$ is a well defined element of $l^\oo (\mathbb{\mathbb{Z}_+})$ if and only if $\lambda \in E$.

\medskip

\no \textbf{Claim 2:}  $\lambda$ is an eigenvalue of $S$ if and only if $(v_n)_{n=0}^{+\oo}$ is a well defined element of $l^\oo (\mathbb{\mathbb{Z}_+})$.
Moreover if  $\lambda$ is an eigenvalue of $S$  then $(v_n)_{n=0}^{+\oo}$ is, up to multiplication by a constant, the unique right eigenvector of $S$ with eigenvalue $\lambda$.

\medskip

\no \textbf{Proof of Claim 1:} If $\lambda \in E$ then $\tilde{f}_{r-1}(\lambda)$ is uniformly bounded and according
to Lemma \ref{bound}, for all $r\geq 1$
$$
\tilde{f}_{r-1}(\lambda)\in \overline{\mathbb{D} (1-p_{r},p_{r})}.
$$
Since $h_r$ maps  $\overline{\mathbb{D} (1-p_{r},p_{r})}$ on $\overline{\mathbb{D} (0,1)} $ we deduce that
$|q_{\lambda}(r)|\leq1 $ and $(|q_{\lambda}(r)|)_{r=1}^{+\oo}$ is bounded above by one.
Indeed, assume that there exists $j_0 \in \mathbb{N}$ such that $|q_{\lambda}(j_0)|>1$. Then $\lambda \not \in E$. Thus  $|\tilde{f}_{r}(z)|>1$ for some $r>1$. We deduce  by (\ref{julia1}) that $\lim_{j \to +\infty} \vert  \tilde{f}_{j}(z) \vert= +\infty$. Hence $\lim_{j \to +\infty}|q_{\lambda}(j)|)= +\infty$.

Conversely, suppose $|q_{\lambda}(r)|\leq 1$ for all $r$. 
From (\ref{LowerBound}) we know that for any $|z|>1$
$$
|h_r(z)| \ge |z|. \, 
$$
Thus, if $|\tilde{f}_{r-1}(\lambda)| > 1$ for some $r > 0$, we have
$$
|q_{\lambda}(r)| = |h_r(\tilde{f}_{r-1}(\lambda))| \ge |\tilde{f}_{r-1}(\lambda)| > 1 \, ,
$$
which yields a contradiction to the fact that $|q_{\lambda}(r)|\leq 1$. Hence $|\tilde{f}_{r-1}(\lambda)|\leq 1$ for all $r$ and, by definition, $\lambda \in E$.

Hence $(|q_{\lambda}(j)|)_{j=1}^{+\oo}$ is bounded above by one if and only if $\lambda \in E$.

\medskip

\no \textbf{Proof of Claim 2:} Let $v=(v_n)_{n\ge 0}$ be a sequence of real numbers and suppose that $(S \, v)_n = \lambda \, v_n$ for every $n \ge 0$. We shall prove that $v$ satisfies (\ref{ps1}). The proof is based on the following representation
\begin{eqnarray}
\label{ps4}
(S \, v)_n & = & \left( \prod_{j=1}^{\zeta_{n}} p_j \right) \, v_{n+1} + (1-p_1) v_n \nn \\
& & + \sum_{r=1}^{\zeta_{n}-1}
\left( \prod_{j=1}^{r} p_j \right) (1-p_{r+1}) v_{n-\sum_{j=1}^r (d-1)d^{j-1}} \, ,
\end{eqnarray}
for $\zeta_{n} \geq 2$
and $(S \, v)_n = p_1 v_{n+1} + (1-p_1) v_n$ if $\zeta_{n}= 1$. This representation follows directly from
the definition of the transition probabilities in (\ref{tp1}). From (\ref{ps4}), we show (\ref{ps1}) by induction.

Indeed, for $n=1$ we have that
$$
\lambda v_0 = (1-p_1)v_0 + p_1 v_1 \ \Rightarrow \ v_1 = \left( \frac{\lambda - (1-p_1)}{p_1}\right) v_0 = q_{\lambda}(1) \, v_0 \, .
$$
Now fix $n \ge 1$ and suppose that (\ref{ps1}) holds for every $1 \le j \le n$. By (\ref{ps4}), since $(S \, v)_n = \lambda \, v_n$, we have that
\begin{equation}
\label{ps5}
\frac{v_{n+1}}{v_0 \, \prod_{r=\zeta_{n}+1}^{\left\lfloor \log_d(n) \right\rfloor + 1} (q_{\lambda}(r))^{a_{r}(n)}}
\end{equation}
is equal to
\begin{eqnarray}
\label{ps6}
\lefteqn{\!\!\!\!\!\!\!\!\!
\frac{ [\lambda - (1-p_1)] \Big[ \prod_{r=1}^{\zeta_{n}-1} (q_{\lambda}(r))^{d-1} \Big] (q_{\lambda}(\zeta_{n}))^{a_{\zeta_{n}}(n)} }
{ \prod_{j=1}^{\zeta_{n}} p_j }
} \nn \\
& & + \frac{ (1-p_2) \Big[ \prod_{r=2}^{\zeta_{n}-1} (q_{\lambda}(r))^{d-1} \Big] (q_{\lambda}(\zeta_{n}))^{a_{\zeta_{n}}(n)} }
{ \prod_{j=2}^{\zeta_{n}} p_j }
\nn \\
& & ... \, + \frac{1-p_{\zeta_{n}}}{p_{\zeta_{n}}} (q_{\lambda}(\zeta_{n}))^{a_{\zeta_{n}}(n)} \, .
\end{eqnarray}
Since
$$
q_{\lambda}(1) = \frac{\lambda - (1-p_1)}{p_1} \, ,
$$
the first term in (\ref{ps6}) is equal to
$$
\frac{ q_{\lambda}(1)^d \, \Big[ \prod_{r=2}^{\zeta_{n}-1} (q_{\lambda}(r))^{d-1} \Big] (q_{\lambda}(\zeta_{n}))^{a_{\zeta_{n}}(n)} }
{ \prod_{j=2}^{\zeta_{n}} p_j } \, .
$$
Summing with the second term we get
$$
\left( \frac{ (q_{\lambda}(1))^d - (1-p_2)}{p_2} \right)
\frac{ \Big[ \prod_{r=2}^{\zeta_{n}-1} (q_{\lambda}(r))^{d-1} \Big] (q_{\lambda}(\zeta_{n}))^{a_{\zeta_{n}}(n)} }
{ \prod_{j=3}^{\zeta_{n}} p_j } \, ,
$$
which is equal to
$$
\frac{ (q_{\lambda}(2))^d \, \Big[ \prod_{r=3}^{\zeta_{n}-1} (q_{\lambda}(r))^{d-1} \Big] (q_{\lambda}(\zeta_{n}))^{a_{\zeta_{n}}(n)} }
{ \prod_{j=3}^{\zeta_{n}} p_j } \, .
$$
By induction we have that the sum of the first $\zeta_{d.n}-1$ terms in (\ref{ps6}) is equal to
$$
\frac{ (q_{\lambda}(\zeta_{n}-1))^d \, (q_{\lambda}(\zeta_{n}))^{a_{\zeta_{n}}(n)} }
{ p_{\zeta_{n}} } \, .
$$
Finally, summing the previous expression with the last term in (\ref{ps6}) we have that (\ref{ps5}) is equal to
$$
\frac{ (q_{\lambda}(\zeta_{n}-1))^d - (1-p_{\zeta_{n}}) }{ p_{\zeta_{n}} } \,
(q_{\lambda}(\zeta_{n}))^{a_{\zeta_{n}}(n)} = (q_{\lambda}(\zeta_{n}))^{a_{\zeta_{n}}(n)+1} \, ,
$$
Therefore,
\begin{eqnarray*}
v_{n+1} & = & v_0 \, (q_{\lambda}(\zeta_{n}))^{a_{\zeta_{n}}(n)+1} \, \prod_{r=\zeta_{n}+1}^{\left\lfloor \log_d(n) \right\rfloor + 1} (q_{\lambda}(r))^{a_{r}(n)} \\
& = & v_0 \prod_{r=1}^{\left\lfloor \log_d(n+1) \right\rfloor + 1} (q_{\lambda}(r))^{a_{r}(n+1)}  \, ,
\end{eqnarray*}
which, by induction, completes the proof of Claim 2.
\end{proof}

\bigskip
%

Denote by $\sigma_{\bar{p}}$ the spectrum of $S_{\bar{p}}$ in $l^\oo(\mathbb{Z}_+)$. We have proved in the previous Proposition that $\sigma_{\bar{p}} \supset E_{\bar{p}}=E$. In the next Proposition, we show that $\sigma_{\bar{p}} \subset E_{\bar{p}}$.

\bigskip

\begin{theorem}\label{spectrum}
The spectrum of $S_{\bar{p}}$ is equal to $E$.
\end{theorem}

\begin{proof}
We prove here that $\sigma_{\bar{p}} \subset E$. Let us denote by $\tau:\mathbb{Z}_+ \ra \mathbb{Z}_+$ the  map $\tau(n)=n+1$ and $\bar{p}_n := (p_{n+j})_{j=0}^\oo$ for
a given $n\in\N$.

Denote $\tilde{S}_{\bar{p}}$ the operator
$$
\tilde{S}_{\bar{p}} := \frac{S_{\bar{p}} - (1-p_1)Id}{p_1} \, ,
$$
which is also a stochastic operator acting on $\mathbb{Z}_+$.  It is associated to an irreducible  Markov chain with period $d$. Indeed let $(Y_n)_{n \ge 0}$ be a Markov chain having $\tilde{S}_{\bar{p}}$ as its transition operator. Then $(Y_n)_{n \ge 0}$ evolves as an $\rm{AMFC}_d$ Markov chain with $p_1 = 1$, thus the chain cannot remain at the same state during two consecutive times. Moreover, $\sum_{j=1}^r (d-1) d^{j-1}$ is equal to $d-1$ modulus $d$ for every $r \ge 1$, then, from (\ref{tp1}), we have that $Y_{n+1}$ is always equal to $Y_n + 1$ modulus $d$. Thus $(Y_n)_{n \ge 0}$ can only return to its starting state at times that are multiples of $d$. Since the probability of $\{Y_d = 0\}$ given $\{Y_0 = 0\}$ is $(1-p_2)>0$, we have that $(Y_n)_{n \ge 0}$ has period $d$.

In particular $\tilde{S}_{\bar{p}}^d$ has $d$ recurrent communication classes that are
$$
C_n= \{ \ j\in \mathbb{N} \ : \ j\equiv n \mod d \} , \quad 0\le n \le d-1 \, .
$$
We now show that $\tilde{S}_{\bar{p}}^d$ acts on each of these classes as a copy of $S_{\bar{p}_2}$ 
(See Remark \ref{rem:per-d} for a rather intuitive justification).

Indeed, for all  $0\le n \le d-1$, we consider the map $\Pi_n: l^{\infty} \to l^{\infty} $  defined
for all $v= (v_{k})_{k \geq 0} \in l^{\infty}$ by
$$(\Pi_n (v))_j= v_k \mbox { if } j= n+ kd \in C_n,\; k \in \mathbb{Z}^+ \mbox { and } 0 \mbox { otherwise }.$$

\begin{lemma}
\label{ikk} With the above notations,
the following properties  hold:
\begin{enumerate}
\item
$\tilde{S}_{\bar{p}} \circ \Pi_{0} = \Pi_{d-1}\circ S_{\bar{p}_2}  \mbox { and } \tilde{S}_{\bar{p}}\circ \Pi_{n} = 
\Pi_{n-1} ,\; \forall\  1 \leq n \leq d-1. $
\item
$\tilde{S}_{\bar{p}}^ {d}= \sum_{i=0}^{d-1} \Pi_i \circ S_{\bar{p}_2} \circ F_i$,
where $F_{i}(v)= (v_i, v_{i+d}, \ldots, v_{i+nd}, \ldots )$ for all $v \in l^{\infty}$ and $i \in \{0,1,\ldots, d-1\}$.
\item
$\sigma (\tilde{S}_{\bar{p}}^ {d})= \sigma ( S_{\bar{p}_2}).$

\end{enumerate}

\end{lemma}

\begin{proof}

Let $v= (v_k)_{k \geq 0} \in l^{\infty},\;i \in \mathbb{Z}^{+}$ and $n \in \{0,\ldots, d-1\}.$
We have
\begin{eqnarray}
\label{tildeS}
(\tilde{S}_{\bar{p}} \circ \Pi_{n} v)_i=  \sum_{j=0}^{+\infty} (\tilde{S}_{\bar{p}})_{i,j}(\Pi_{n}v)_j=  \sum_{k=0}^{+\infty}(\tilde{S}_{\bar{p}})_{i,n+kd} v_k.
\end{eqnarray}
For all $1\leq n  \leq d-1$,  $(\tilde{S}_{\bar{p}})_{i,n+kd}=  1$ whenever  $i= n-1+ kd$ and $0$ otherwise. 
Hence $\tilde{S}_{\bar{p}} \circ \Pi_{n} = \Pi_{n-1}.$ \\
If $ n=0,\; (\tilde{S}_{\bar{p}})_{i,n+kd}= (\tilde{S}_{\bar{p}})_{i,kd} \ne 0.$\\ 
{\bf - Case 1}:  $i= (d-1)+ md,\; m \in \mathbb{N}$.
\\
By (\ref{tp1}),  $(S_{\bar{p}})_{(d-1)+ md,kd}= p_1  (S_{\bar{p}_{2}})_{m,k}$ then 
\begin{eqnarray}
\label{zs}
(\tilde{S}_{\bar{p}})_{(d-1)+ md,kd}=   (S_{\bar{p}_2})_{m,k}.
 \end{eqnarray}
 By (\ref{tildeS}) and (\ref{zs}), we deduce that
 $(\tilde{S}_{\bar{p}} \circ \Pi_{0})_i = (\Pi_{d-1}\circ S_{\bar{p}_2})_i.$\\
 {\bf - Case 2}:  $i \not \in (d-1)+ d \mathbb{Z}^{+}$.\\
Then  $(S_{\bar{p}})_{i,kd}=0$ for all $k$. Hence 
$(\tilde{S}_{\bar{p}} \circ \Pi_{0})_i =
0= (\Pi_{d-1}\circ S_{\bar{p}_2})_i$.
Thus $\tilde{S}_{\bar{p}} \circ \Pi_{0} = \Pi_{d-1}\circ S_{\bar{p}_2}.$ 
4
 As regard item 2) of Lemma~\ref{ikk}, from  1) we can easily prove that for all integers $0 \leq n \leq d-1 $
 and $1 \leq i \leq d$, we have
 $$
\tilde{S}_{\bar{p}}^ {i} \circ \Pi_n = \left\{
\begin{array}{cl}
\Pi_{n-i} & \mbox { if }   i \leq n \\
\Pi_{d+n-i} \circ S_{\bar{p}_2} & \mbox { if }    i > n \, .
\end{array}
\right.
$$
Hence
 \begin{eqnarray}
 \label{rrp}
 \tilde{S}_{\bar{p}}^ {d} \circ \Pi_n=  \Pi_n \circ S_{\bar{p}_2},\; \forall 0 \leq n \leq d-1  .
 \end{eqnarray}
  Using (\ref{rrp}) and the fact that $Id= \sum_{i=0}^{d-1} \Pi_i \circ F_i$, we obtain of 2) Lemma~\ref{ikk}.

 It remains to show the last item. By 2), we have
 \begin{eqnarray}
 \label{xxc}
 (\tilde{S}_{\bar{p}}^ {d} - \lambda Id)= \sum_{i=0}^{d-1} \Pi_i \circ ( S_{\bar{p}_2} - \lambda Id) \circ F_i,
 \end{eqnarray}
 for $\lambda \in \mathbb{C}$.\\
 {\bf Claim: } $(\tilde{S}_{\bar{p}}^ {d} - \lambda Id)$ is bijective if and only if 
 $( S_{\bar{p}_2} - \lambda Id)$ has the same property. 
 We will show that these 2 maps are both injective and onto for the same values of $\lambda$.
 
 Indeed, assume there exists 
 $v= (v_i)_{i \geq 0} \in l^{\infty} \setminus \{0\}$ such that $(S_{\bar{p}_2} - \lambda Id)(v)=0$.
 We consider $w= (w_{i})_{i \geq 0}$ defined by
 $w_i= v_{[i/d]}$ where $[i/d]$ denotes the integer part of $i/d$.
 We have $ w
 = (\underbrace{v_0,\ldots v_0}_{d},\ldots, \underbrace{v_i \ldots v_i}_d, \ldots ) \in l^{\infty}\setminus \{0\}$ and $(\tilde{S}_{\bar{p}}^ {d} - \lambda Id)w= \sum_{i=0}^{d-1} \Pi_i \circ ( S_{\bar{p}_2} - \lambda Id) (v)= 0.$ 
 
 On the other hand, if  $(\tilde{S}_{\bar{p}}^ {d} - \lambda Id) (u)= 0$ for some $u= (u_{i})_{i \geq 0} \in l^{\infty}\setminus \{0\}$. By (\ref{xxc}), we have
 $( S_{\bar{p}_2} - \lambda Id) \circ F_i (u)= 0$ for all $i \in \{0,\ldots d-1\}$. 
 Hence $(\tilde{S}_{\bar{p}}^ {d} - \lambda Id)$ is one to one if and only if $(S_{\bar{p}_2} - \lambda Id)$
 is.
 
 Now, suppose that $(\tilde{S}_{\bar{p}}^ {d} - \lambda Id)$ is onto.
 Let $v= (v_i)_{i \geq 0} \in l^{\infty} $ and $v'
 = (\underbrace{v_0,\ldots v_0}_{d},\ldots, \underbrace{v_i \ldots v_i}_d, \ldots )$.
 Then there exists $u \in l^{\infty}$ such that  $(\tilde{S}_{\bar{p}}^ {d} - \lambda Id) (u)= v'$.
 Hence $(S_{\bar{p}_2} - \lambda Id) \circ F_i (u)=v$ for all $i=0,\ldots, d-1$. Thus $S_{\bar{p}_2} - \lambda Id$ is onto.
 
 Now assume  that $S_{\bar{p}_2} - \lambda Id$ is onto. Let $v= (v_i)_{i \geq 0} \in l^{\infty} $ 
 and  $u^{(k)},\; k=0,\ldots, d-1$, elements of $l^{\infty}$ such that 
  $(S_{\bar{p}_2} - \lambda Id)(u^{(k)})= F_k (v)$.
  Let $w$ be the unique element in  $l^{\infty}$ such that $F_{k}(w)= u^{(k)}$ for all $ k \in \{0,\ldots, d-1\}$.
  Hence $(\tilde{S}_{\bar{p}}^ {d} - \lambda Id)(w)= \sum_{i=0}^{d-1} \Pi_i \circ F_i (v)= v$.
  Thus   $(\tilde{S}_{\bar{p}}^ {d} - \lambda Id)$ is onto and we obtain the claim.

\end{proof}

\noindent
{\bf End of the  proof of Theorem \ref{spectrum}}.
By item 3 of Lemma  \ref{ikk}, the spectrum of $\tilde{S}_{\bar{p}}^d$ is equal to the spectrum of $S_{\bar{p}_2}$. Since, $\tilde{S}_{\bar{p}}^d = \tilde{f}_{1}\big( S_{\bar{p}} \big)$, by the Spectral Mapping Theorem(see \cite{ds}),  we have that
$$
\tilde{f}_{1}\big( \sigma_{\bar{p}} \big) = \sigma_{\bar{p}_2} \, .
$$
By induction, we deduce that
$$
\tilde{f}_{j+1}\big( \sigma_{\bar{p}} \big) = \sigma_{\bar{p}_{j+1}} \, ,
$$
for every $j\ge 1$. Since $S_{\bar{p}_{j+1}}$ is a stochastic operator, its spectrum is a subset of $D(0,1)$. Therefore
$$
| \tilde{f}_{j+1}\big( \lambda \big) | \le 1 \, ,
$$
for every $j$ and $\lambda \in \sigma_{\bar{p}}$. This implies that $\sigma_{\bar{p}} \subset E$.
\end{proof}

\smallskip

\begin{remark} \label{rem:per-d}
 let us give a rather intuitive justification to the fact that $\sigma (\tilde{S}_{\bar{p}}^ {d})= \sigma ( S_{\bar{p}_2}).$ Recall that $\tilde{S}_{\bar{p}}$ is the transition operator of an $\rm{AMFC}_d$ Markov chain with $p_1 = 1$. Each time we apply $\tilde{S}_{\bar{p}}$ to a given $n \ge 1$, we add one to $a_1(n)$, the first digit of $n$ in its $d$-adic expansion. Iterating $\tilde{S}_{\bar{p}}$ $d$-times, we necessarily end at a non-negative integer $m$ with $a_1(m)=a_1(n)$, i.e., $n = m$ modulus $d$. Moreover, at one single iteration, among these $d$ iterations of $\tilde{S}_{\bar{p}}$, we have the first digit of the outcome equal to $d-1$, and $\tilde{S}_{\bar{p}}$ will act on the shifted sequence of digits $(a_j(n))_{j\ge 2}$ as an $\rm{AMFC}_d$ algorithm associated to $\bar{p}_2$. Therefore, after the $d$ iterations of $\tilde{S}_{\bar{p}}$, equivalent to one iteration of $\tilde{S}_{\bar{p}}^d$, we obtain that the outcome $m$ have $a_1(m)=a_1(n)$ and $(a_j(m))_{j\ge 2}$ randomly generated by $(a_j(n))_{j\ge 2}$ through the $\rm{AMFC}_d$ algorithm associated to $\bar{p}_2$. 
 \end{remark}

\begin{remark}
Let $(X(t))_{t \ge 0}$ be an irreducible recurrent $\rm{AMFC}_d$ Markov chain starting at $0$.
Suppose that $\lambda \in \mathbb{R} \cap E$, $\lambda \neq 1$ and $v = (v_n)_{n\ge 0}$ is an eigenvector associated to $\lambda$ with $v_0 > 0$. Since $E \subset D(0,1)$, then $v$ is a super-harmonic function on $l^\infty (\mathbb{Z}_+)$ with respect to $S$, i.e., $(Sv)_n \le v_n$, for all $n \ge 0$. Therefore, by (\ref{ps1}) and from the Potential Theory for Markov chains, we have that $(v_{X(t)})_{t \in \mathbb{Z}_+}$ is a bounded supermartingale. If $(v_{X(t)})_{t \in \mathbb{Z}_+}$ is positive, by the convergence theorem for supermartingales, we have that there exists a bounded $\mathbb{Z}_+$ valued random variable $v_\infty$ such that $\lim_{t \rightarrow \infty} v_{X(t)} = v_\infty$ almost surely. Since the chain is irreducible and recurrent, this can happen only if $v_\infty$ is constant. But for $\lambda \neq 1$,  by (\ref{ps1}), $v_\infty$ is non-constant. Therefore, either
$\mathbb{R} \cap E = \{1\}$ or $v_n < 0$ for some $n$. However we have $v_n < 0$ for some $n$ if and only if $\tilde{f}_{j}(\lambda) < 0$ for some $j\ge 1$, which does not happen for a constant sequence $p_i = p$, with $p \ge \frac{1}{2}$ (by Lemma \ref{bound} $\tilde{f}_{j}(\lambda) \in \bar{\mathbb{D}}(1/2,1/2)$). Thus if $p \ge \frac{1}{2}$, $\mathbb{R} \cap E = \{1\}$.
\end{remark}

\smallskip

\section{Connectedness properties of the spectra}
\label{sec:connectedness}

This section is devoted to the study of the connectedness of $E$ and its complement. According to the parity of $d$, we obtain conditions for $E$ to be connected, to have a finite number of connected components or to be a Cantor set.

\smallskip

\begin{proposition}
\label{propjulia}
The set $\mathbb{C} \setminus E$ is connected.
 \end{proposition}
 \begin{proof}

By Lemma \ref{bound} and Corollary~\ref{cor:bound}, we have $E= \bigcap_{n=1}^{+\infty} \tilde{f}_n ^{-1} \big( \overline{\mathbb{D}(0,1)} \big)$ where $\tilde{f}_{n+1}^{-1} \big(\overline{\mathbb{D} (0,1)} \big) \subset \tilde{f}_{n} ^{-1} \big( \overline{\mathbb{D}(0,1)} \big)$ for every integer $n \geq 1$. Then
  $\mathbb{C} \setminus  E= \bigcup_{n=1}^{+ \infty} \mathbb{C} \setminus  \tilde{f}_n ^{-1} \big( \overline{\mathbb{D}(0,1)} \big).$
  Since for any integer $n, f_n$ is a polynomial function, then $ X_n= \mathbb{C} \setminus \tilde{f}_n ^{-1} \big( \overline{\mathbb{D}(0,1)}\big)$ is a connected set.
  Since $X_n$ is an increasing sequence, we deduce that $\mathbb{C} \setminus E$ is also a connected set.
\end{proof}

\smallskip

We need to introduce some new notation. Let $g_{j}:  \mathbb{C} \ra \mathbb{C},\; j \ge 2$, be the function defined by
$$
g_{j}(z) := \frac{1}{p_{j}}z^d -\frac {(1-p_j)}{p_j},
$$
and  $\tilde{g}_{j} := g_{j+1} \comp ... \comp g_{2}$.
Let
\begin{eqnarray}
\label{E}
K = \Big\{ z \in \mathbb{C} : \limsup_{j\ra +\oo} |\tilde{g}_{j}(z)| < +\oo \Big\}.
\end{eqnarray}

 Recall that $h_j(z)=\frac{z}{p_j}-\frac{1-p_j}{p_j}$.
 The functions $g_{j+1}$ and $f_j$ are conjugated in the following sense
\begin{equation}
\label{conjugacy}
g_{j+1}=h_{j+1}\circ f_j\comp h_j^{-1}.
\end{equation}
Since
 $\tilde{g}_{j} (h_1(z))= h_{j+1}\comp \tilde{f}_{j} (z)$ for all $j \geq 2$,
  we deduce that
 \begin{eqnarray}
\label{if}
  \mbox { if } \liminf_{i \ra \infty} p_i >0, \mbox { then }  E := h_1^{-1} (K) \, .
  \end{eqnarray}

According to Lemma \ref{bound}, since $h_1$ maps the disk $\overline{\mathbb{D}(1-p_1,p_1)}$ to $\overline{\mathbb{D}(0,1)}$, we have that if $ \liminf_{i \ra \infty} p_i >0$, then  $K$ is also included
in the closed disk $\overline{\mathbb{D}(0,1)}$. Indeed, by the Lemma \ref{bound}, we have an analog of Corollary \ref{cor:bound}.
 In  the case where $\liminf_{i \ra \infty} p_i =0$,  the same results are true, and we have $h_1^{-1} (K) \subset E$.

\smallskip

\begin{lemma}
\label{RR}
Let $R >1$   then $K= \bigcap_{n=2}^{+\infty} \tilde{g}_{n}^{-1} (\mathbb{D}(0,R))$ where
$\tilde{g}_{n+1} ^{-1} (\mathbb{D}(0,R)) \subset \tilde{g}_{n} ^{-1} (\mathbb{D}(0,R))$ for every integer $n \geq 2$.

\end{lemma}

For the study of connectedness of $K$, we will use the Riemann-Hurwitz-Formula (see ~\cite{Mil} for the formula
in the general  case and~\cite{ste} for a short proof in the case of plane domains).

\textsc{ Riemann-Hurwitz-Formula} : Let $V$ and $W$ be 
domains on the Riemann sphere and let $f : V\rightarrow W$ be a branched covering map of degree $d$ with $r$ 
critical points (counted with multiplicity). Then
$$\chi(V) = d\chi(W) - r,
$$
where $\chi(V)$ and $\chi(W)$ are the Euler characteristic of $V$ and $W$, respectively. 

This imply, if $V$ and $W$ are domains of $\C$ of connectivity $n$ and $m$ respectively that
$$m-2= d(n-2)+ r.$$
In particular, if $W$ is connected and $r=d-1$ then $V$ is connected.

\begin{remark}
\label{remark:RR}
\textbf{1.} The use of $\tilde{g}_j$ instead of $\tilde{f}_j$ is to simplify the study of some topological properties of $E$. Indeed, since $h_1$ is a linear homeomorphism from $\mathbb{C}$ to $\mathbb{C}$, we have when $\liminf_{i \ra \infty} p_i >0 $ that $E$ and $K$ are (really) the same up to a linear change of coordinates. \\
\textbf{2.} The analysis of the number of connected components of $E$ relies on the Riemann-Hurwitz Formula.  
It relates the number of connected components of $E$ (resp. $K$) with the number of critical points (counted with multiplicity) of $\tilde{f}_{n}$ (resp. $\tilde{g}_{n}$)  that do not belong to $E$ (resp. $K$). The critical points of $\tilde{f}_{n}$ (resp. $\tilde{g}_{n}$) are of the form $z \in \mathbb{C}$ such that
$z=1-p_1$ (resp. $z=0$) or $\tilde{f}_{k}(z)=1-p_{k+1}$ (resp. $\tilde{g}_{k}(z) = 0$) for some integer $ 1 \leq k < n$.

\end{remark}

\medskip

\subsection{Case where $p_i= p$ for all $i \geq 2$}

Here $g_j$ does not depend on $j$ and will be denoted by
$$
g(z) := \frac{1}{p}z^d -\frac {(1-p)}{p} \, , \ z \in \mathbb{C}.
$$
The polynomial map $g$ has a unique critical point at 0. Moreover, $g(1)=1$, and 1 is a repelling fixed point of $g$.

\begin{lemma}
\label{0inK}
Suppose that $p_i= p$ for all $i \geq 2$. If $0 \in K $ then $K$ is a connected set, otherwise $K$ is a Cantor set.
\end{lemma}

\medskip

\begin{proof}
The main point of this  lemma is that $g$ is a uni-critical polynomial with its critical point at $0$.\\
If $0\in K$, then all  critical points of $g$ belong to $K$ and by Lemma~\ref{RR} and  Riemann-Hurwitz Formula, we deduce that $K$ is connected.\\
Otherwise, $0\notin K$ then all  critical points of $g$ belong to the immediate bassin of $\infty$. By a classical
result of complex dynamics, see \cite{bv} Theorem 6.4 and also \cite{MNTU} Theorem 1.1.6, $g$ is a hyperbolic polynomial and its Julia set
is a Cantor set.
 
\end{proof}
\smallskip

\begin{proposition}
\label{prop:mudreg}
Let $p \in (0,1)$ be a fixed real number and suppose that $p_j = p$ for all $j \ge 2$, then the following properties are satisfied:
\begin{enumerate}
  \item[(i)]
  If $d$ is even, then
$E$ is connected if and only if  $p \geq \frac{1}{2}$, otherwise $E$ is a Cantor set.
  \item[(ii)]
    If $d$ is odd, then
$E$ is connected, if and only if, $p \geq \vartheta_d = d \theta_d^{d-1}$ where $\theta_d \in (0,1)$ is the
unique non-negative solution of
\begin{equation}
\label{eq:mudreg}
d x^{d-1} + (d-1) x^d - 1=0 \, ,
\end{equation}
otherwise $E$ is a Cantor set.
\end{enumerate}
\end{proposition}

\medskip

\begin{remark}
\textbf{1.} Since $\theta_d \in (0,1)$, we have that $d \, \theta_d^{d-1} > (d-1) \, \theta_d^d$ and by equation
(\ref{eq:mudreg}) we have that $\vartheta_d > 1/2$. This implies a noticeable difference with the case
$d$ even. \\
\textbf{2.} If $d = 3$ then $\theta_3 = 1/2$ and $\vartheta_3 = 3/4$. \\
\textbf{3.} We have that
$\vartheta_d$ decreases to $1/2$ as $d \rightarrow \infty$.
\end{remark}
%

\begin{proof}\emph{of Proposition \ref{prop:mudreg}.}
\medskip

(i) Assume that $d$ is even.
If $p < \frac{1}{2}$, then $g(0)= -\frac{(1-p)}{p} <-1$, hence $ 0 \not\in K$. By Lemma \ref{0inK} and Remark \ref{remark:RR}, $K$ and $E$ are Cantor sets. Now, suppose $p \geq  \frac{1}{2}$  and let $-1 \leq x \leq 1$, then we see easily by induction on $n$ that $ -1 \leq -\frac{(1-p)}{p} \leq g^{n}(x) \leq 1$ for all $n \in \mathbb{N}$. Hence $ 0  \in K$. By Lemma \ref{0inK} and Remark \ref{remark:RR}, $K$ and $E$ are connected.

\medskip

(ii) Now, assume that $d$ is odd.
\bigskip

\no \textbf{Claim :} $0 \in K$ if and only if the equation $x= g(x)$ has a real solution $ -1 \leq x \leq 0$.

\medskip

\no \textit{Proof of the Claim:} Assume that there exists a real number $x  \leq 0$  such that $x= g(x)$. Then, from monotonicity properties of $g$, for any integer $n \geq 0,$ we have  $x \leq g^{n}(0)  \leq 0$, therefore $0 \in K$.

Now suppose that $0 \in K$ and
put $l= \inf \{g^{n}(0): \; n \in \mathbb{N}\}$.
Let $\varepsilon  >0$ and $n \in \mathbb{N}$ such that $g^{n}(0)= \frac{1}{p} g^{n-1}(0)^{d} - \frac{1-p}{p} < l+ \varepsilon $,
then $g(l) = \frac {l^d - (1-p)}{p} < l+ \varepsilon $.
On the other hand $l \leq g^{n+1}(0) < \frac{1}{p} ( l+ \varepsilon )^d - \frac{1-p}{p}.$
Since $\epsilon$ is arbitrary, then $l= g(l)= \frac {l^d - (1-p)}{p}.$
From the fact that $-1 \leq l \leq 0$, the claim is proved.

\bigskip

Now consider the equation $\psi (x)= p (g(x)-x)= x^d- px- (1-p).$
Since $\psi (-1)= 2(p-1) <0$, $\psi (0)= p-1 <0$, 
\begin{displaymath}
\max_{x\le 0} \psi (x) \leq \psi \Big (- \big( \frac{\, p \,}{\, d \,} \big)^{\frac{1}{d-1}} \Big)
\end{displaymath}
and $\big( \frac{\, p \,}{\, d \,} \big)^{\frac{1}{d-1}} < 1$, we have that
$$ 0 \in K \Longleftrightarrow  \psi \Big( - \Big( \, \frac{p}{d} \, \Big)^{\frac{1}{d-1}} \Big) \geq 0.$$
Consider the function $\phi(p)=  \psi \big(- (\frac{p}{d} )^{\frac{1}{d-1}} \big)= (d-1) \big( (\frac{p}{d} )^{\frac{d}{d-1}} \big) - (1-p) $.
As  for any integer $d \geq 2$, we have $\phi(0) <0, \phi(1) >0$ and $\phi$ is  increasing, we deduce that there exists a unique real number
$0 < \vartheta_d <1$ such that 
$ \phi (\vartheta_d)= \psi \Big( - \big(\frac{\vartheta_d}{d} \big)^{\frac{1}{d-1}} \Big) = 0$.
Since $\partial_{p} \psi (x) = -x+1 \geq 0$ for all $x \leq 0$, we obtain that $0 \in K$ if and only if $p \geq \vartheta_d$.
On the other hand,
if $\theta_d=  \Big (\frac{\vartheta_d}{d} \Big)^{\frac{1}{d-1}}$, then
$(d-1)\theta_d^d + d \theta_{d}^{d-1}-1=0$
and  $\vartheta_d = d \theta_d^{d-1}$.
\end{proof}

\medskip

\subsection{Case where $(p_i)_{i \geq 2}$ is not constant}
Now we focus on the general setting, where again there are two different behaviors with respect to the parity of $d$. In Proposition \ref{propjulia} we consider the case $d$ even and in Proposition \ref{prop:odd} the case $d$ odd. These Propositions should be considered in analogy respectively with items (i) and (ii) in Proposition \ref{prop:mudreg}. The main difference with the case $(p_i)_{i \geq 2}$ constant is that in the general case $E$ may be non-connected and have a finite number of connected components. 
%

\begin{proposition}
\label{propjulia}
Assume that $d$ is even and let $s= \# \{i \geq 2,\; p_i < \frac{1}{2}\}$. Then we have the following results.
\begin{enumerate}
\item[(i)]
If $s=0$, then $E$ is connected.
\item[(ii)]
If  $0 < s <+ \infty $, then $E$ has $d^k$ connected components where $s \leq k \leq t-1$ where $p_t < \frac{1}{2}$ and $p_i \geq \frac{1}{2}$ for all $i > t$. In particular, if $p_i <1/2$ for all $2 \leq i \leq n$ and $p_i \geq 1/2$ for all $i >n+1$, then $E$ has exactly $d^{n-1}$ connected components.
\item[(iii)]
If $s= +\infty$, then $E$ has an infinite number of connected components.
\end{enumerate}
\end{proposition}

 \begin{proof}

 \medskip


(i) Let $n \geq 2$ and $z$ be a critical point of $\tilde{g}_{n}= g_{n+1} \circ \ldots \circ g_2$, then $z=0$ or  $\tilde{g}_{i}(z) = 0$ for some $2 \leq
 i \leq n$.

Suppose $s=0$ then $p_i \geq \frac{1}{2}$ for all integers $i \geq 2$. It is easy to see that all integers $2 \leq k \leq m,\;
 -1  \leq 1-\frac{1}{p_{m+1}} \leq  \tilde{g}_{m}(0) \leq 1$.
We deduce that all critical points of $\tilde{g}_n$ belong to $ K$. Hence $K$ is a connected set. Thus by (\ref{if}), we deduce that $E$ is connected,

\medskip

(ii) Assume that $0 < s <+ \infty $, then there exist $s$ integers $k_1 < k_2 < \ldots < k_s$  such that $p_{k_{j}} < \frac{1}{2}$ for all $ 1 \leq j \leq s$. Hence
$1- \frac{1}{p_{k_{j}}} < -1$ for all $1 \leq j \leq s$.

Let $ n >k_s,\; 1 \leq j \leq s$,  and  $z_j $ be a complex number such that $\tilde{g}_{k_j -2}(z_j) = 0$, then
$\tilde{g}_{n}(z_j)= g_{n+1} \circ \ldots \circ g_{k_{j}+1}( 1- \frac{1}{p_{k_{j}}})$  diverges to $-\infty$ as $n \ra \infty$. Hence $z_j \not \in K$.
Since  $z_j$ is a critical point of $\tilde{g}_{n}$, using the Riemann-Hurwitz formula, we deduce that $K$ has at least $d^s$ connected components. On the other hand,
as $p_i \geq \frac{1}{2}$ for all $i > k_s$, we deduce that any complex number $z $ satisfying  $\tilde{g}_m (z)=0,\; m >  k_s$, belongs to $K$. Therefore $E$ has at most $d^{k_s}$ connected components.
\medskip

(iii) Suppose that $s= +\infty$. If $\liminf_{i \ra \infty} p_i >0$, then by (\ref{if}) and item (ii), we deduce that $E$ has an infinite number of connected components.

Now, suppose that $\liminf_{i \ra \infty} p_i =0$. Since if $R >1$ is a real number, then  $E= \bigcap_{n=1}^{+\infty} \tilde{f}_n ^{-1} \mathbb{D}(0,R)$ where $\tilde{f}_{n+1} ^{-1} \mathbb{D}(0,R) \subset \tilde{f}_{n} ^{-1} \mathbb{D}(0,R)$ for every integer $n \geq 1$; and the fact that the critical points of $\tilde{f}_n $ are $z$ such that $z=1-p_1$ or $\tilde{f}_k(z)= 1-p_{k+1}$ for some integer $k <n$. Using the same idea used in the proof of item (ii), we deduce that $E$ has an infinite number of connected components.

\end{proof}
\medskip

\begin{proposition}
\label{prop:odd}
If $d$ is odd, then the following assertions hold:
\begin{enumerate}
\item[(i)] $E$ is connected if and only if $1-p_1 \in E$.
\item[(ii)] If $p_j \ge \vartheta_d$ for all $j \ge 2$, then $E$ is connected.
\item[(iii)] For every $0 < \delta < \vartheta_d$, there exists $k=k(\delta)$ such that if $p_{m+r} < \vartheta_d - \delta$, for some $m\ge 2$ and all $0\le r \le k$, then $E$ is not connected.
\item[(iv)] If $p_{j} < 1/2$, for some $j$, then $E$ has at least $d^j$ connected components.
\item[(v)] If $p_j < 1/2$ for infinitely many $j$'s, then $E$ has an infinite number of connected components.
\item[(vi)] If $p_j$ is randomly chosen in a way that $(p_i)_{i=1}^\infty$ is a sequence of iid
random variables with $P(p_i < \vartheta_d) > 0$, then
$$P( E_{\bar{p}} \textrm{ has an infinite number of connected components}) = 1.$$
\end{enumerate}
\end{proposition}

\medskip

For the proof of the previous Proposition and the next examples, we cease to mention $K$ and we speak directly of $E$. This avoids some unnecessary assertions.

\medskip

\noindent{\bf Examples :}
We show here with examples in the case where $d$ is odd that little can be said about the connectedness of $E$ when
we have indices $j$ for which $1/2 \le p_j < \vartheta_d$. Recall that in the case $d=3$, we have $\theta_3 = \frac{1}{2}$
and $\vartheta_3 = \frac{3}{4}$.
\begin{enumerate}
\item[1.] Take $p_2 = 2/3$ and $p_j=3/4$ otherwise. In this case, $\tilde{g}_2 (0) = - 1/2 = - \theta_3$ thus $\tilde{g}_j (0) = - 1/2$ for every $j$. Since the sequence $(\tilde{g}_j (0))_{j\ge 2}$ is bounded, $ E$ is connected.
\item[2.] Take $p_2 = 2/3$, $p_3=9/14$ and $p_j=3/4$ otherwise. In this case, $\tilde{g}_2 (0) = - 1/2$ and $\tilde{g}_3 (0) = - \frac{3}{4} < -1/2$,
thus $\lim_{j \rightarrow \infty} \tilde{g}_j (0) = - \infty$ and $E$ is not connected. Moreover, any other critical point of
$\tilde{g}_j$ distinct from $0$ is in $E$ by the previous case. Then $E$ has $d$ connected components by the
Riemann-Hurwitz Lemma (see Figure \ref{non-connected} above).
\item[3.] Take $p_2 = 2/3$, $p_3=9/14$, $p_4=\frac{126}{128}$ and $p_j=3/4$ otherwise. In this case, $\tilde{g}_2 (0) = 1/2$,
$\tilde{g}_3 (0) = - \frac{3}{4} < -1/2$ and $\tilde{g}_4 (0) = - \frac{6}{14} - \frac{1}{63} > -1/2$. Thus $-1/2 < \tilde{g}_j (0) < 0$
for every $j\ge 4$ which implies that $(\tilde{g}_j (0))_{j\ge 2}$ is bounded and $E$ is connected.
\end{enumerate}

\medskip

\begin{figure}[htb]
\begin{center}
\includegraphics[width=10cm]{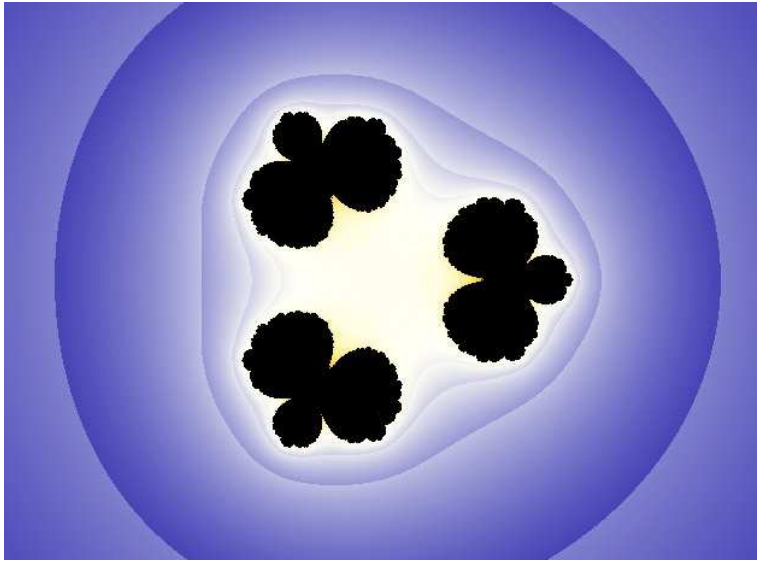}
\caption{In degree 3, $E$ may have a finite number of connected components }
\label{non-connected}
\end{center}
\end{figure}

\begin{proof}\emph{of Proposition \ref{prop:odd}.}

\medskip

\no \textit{Proof of (i):}
By the Riemann-Hurwitz Formula, $E$ is connected, if and only if, $E$ contains all the critical points of
$\tilde{f}_j$ for every $j \ge 1$. Since $1-p_1$ is clearly a critical point of every $\tilde{f}_j$, if $E$ is connected
then $1-p_1 \in E$. Now suppose that $1-p_1 \in E$ and let $z$ be a critical point of $\tilde{f}_l$, for some fixed $l \ge 2$.
By the chain rule for derivatives, there exists $1 \le m < l$ such that $\tilde{f}_m(z) = 1-p_{m+1}$. Thus, for every $j \geq l$,
$$
0 \ge \tilde{f}_j (z) = (f_{j} \comp ... \comp f_{m+2}) (0) \ge
(f_{j} \comp ... \comp f_{m+2}) ( \tilde{f}_{m+1}(1-p_1) ) = \tilde{f}_j (1-p_1)  \, ,
$$
which follows from the fact that  $(f_j \comp ... \comp f_{m})$ is increasing and   $\tilde{f}_{m+1}(1-p_1)= (f_{m+1} \comp ... \comp f_{2}) (0) < 0$.
Since $(\tilde{f}_j (1-p_1))_{j\ge 1}$ is bounded, $(\tilde{f}_j (z))_{j\ge 1}$ is also bounded and $z \in E$.
Therefore $E$ is connected.

\medskip

\no \textit{Proof of (ii):} Since $d$ is odd, $0 \in K$ (equivalent to $1-p_1 \in E$) if and only if
$\lim_{j \ra \infty} \tilde{g}_j (0) = - \infty$. Let
$$
g(x,p) := \frac{x^d - (1-p)}{p},
$$
then $\partial_p g(x,p) \ge 0$, for $x < 0$. Thus
$$
\tilde{g}_j(0) \ge g^{j} (0,\vartheta_d) \, ,
$$
for every $j$. By Proposition \ref{prop:mudreg}, we have  $-1 < \lim_{n \ra \infty} g^j (0,\vartheta_d) < 0$, thus (ii) holds.

\medskip

\no \textit{Proof of (iii):}
By Proposition \ref{prop:mudreg}, we have  $\lim_{j \ra \infty} g^j (0,\vartheta_d - \delta) = -\infty$. Therefore
there exists  $k=k(\delta)>1$ such that
$$
g^{k-1} (0,\vartheta_d - \delta) < -1 \, .
$$
Put
$$
f(x,p) := \left(\frac{x - (1-p)}{p} \right )^d.
$$
By (4.2),  $ h_{\vartheta_d - \delta}  \comp
f^{k-1} (1-p_1,\vartheta_d - \delta) = g^{k-1} (0,\vartheta_d - \delta)$.
Hence $f^{k} (1-p_1,\vartheta_d - \delta) =  \left( h_{\vartheta_d - \delta} \comp
f^{k-1} (1-p_1,\vartheta_d - \delta) \right)^d <-1$.

Now suppose that $p_{m+r} < \vartheta_d - \delta$, for some $m\ge 2$ and all $0\le r \le k$.
Since, $f$ is increasing in both $x$ and $p$, for $x\le 0$, then
\begin{eqnarray*}
\tilde{f}_{m+r} (1-p_1) & = & (f_{m+r} \comp ... \comp f_{m+1})( \tilde{f}_{m} (1-p_1) ) \\
& \le & (f_{m+r} \comp ... \comp f_{m+1})( 0 ) \le f^k (0,\vartheta_d - \delta) \le -1 \, ,
\end{eqnarray*}
for every $j \ge m+r$. Thus $1-p_1 \notin E$ and hence  $E$ is not connected.

\medskip

\no \textit{Proof of (iv):} If $p_j < 1/2$ then $f_j(0) < -1$. For every critical point $z \in \mathbb{C}$
of $\tilde{g}_j$, we have that $\tilde{f}_m(z) = 1-p_{m+1}$ for some $0\le m < j$, then $\tilde{f}_{j+1}(z) < -1$. Thus $z \notin E$.
By the Riemann-Hurwitz Formula, $E$ has at least $d^j$ connected components.

\medskip

\no \textit{Proof of (v):} Take $j \rightarrow \infty$ in (iv).

\medskip

\no \textit{Proof of (vi):} By continuity of measures, there exists $\delta$ such that
$P(p_i < \vartheta_d - \delta) > 0$. Take $k=k(\delta)$ as in (iii).
By the Borel Cantelli Lemma, we have that almost surely
there exists a sequence $(j_n)_{n\ge 1}$ such that $\lim_{n\rightarrow \infty} j_n = \infty$
and $p_{j_n+r} < \vartheta_d - \delta$ for all $0\le r \le k$ and $n \ge 1$. Since, from
(iii), we have
$$
(f_{j_n+r} \comp ... \comp f_{j_n})( 0 ) < -1 \, ,
$$
following the arguments in (iv), we get that $E_{\bar{\rho}}$ has at least
$d^{j_n+r}$ connected components. Since $\lim_{n\rightarrow \infty} j_n = \infty$,
we have that $E_{\bar{\rho}}$ has an infinite number of connected components.
\end{proof}

\medskip

\section{Further topological and analytical properties of the spectra}
\label{sec:fibered}

\smallskip

Fix $d\geq 2$. We are now able to identify the sets $E$ and $K$ with fibered Julia sets associated to a suitable
fibered polynomial  (see \cite{ses}).
Indeed, let us denote $X=[0,1]^\N$, which is a compact set endowed with the product topology. \\
We define the shift map $\tau : X\rightarrow X$ in the standard
way: for all $p=(p_1,\ldots,p_n,\ldots) \in X$, $ \tau(p)=( p_2,\ldots, p_{n+1}\ldots)$.
Let us define the "parameters" of our fibered polynomial map for all $p=(p_n)_{n \geq 1}\in X$ by
$$
a(p)=\frac{1}{p_1} \mbox{ and } b(p)=-\frac{1-p_1}{p_1}.
$$
We introduce the fibered polynomial map of degree $d$ over $(X,\tau)$ defined by
$$
\begin{array}{cccc}
 P\ : & X \times \mathbb{C} & \longrightarrow &
X\times \mathbb{C} \\
{} & (p,z) & \longmapsto & \left(\tau(p),
P_p(z)=a(p)z^d+b(p)\right),
\end{array}
$$
Note that $P$ is a skew-product map ({\it i.e.} the first variable does not depend on the second one).
If $P^n$ is the $n$-iterate of $P$ then $P^n(p,z)=(\tau^n(p), P^n_p(z))$ where
\begin{equation}\label{Pn}
P^n_p(z)= P_{\tau^{n-1}(p)}\circ\ldots \circ P_{\tau(p)}\circ P_p(z) \, .
\end{equation}
Observe that $P^{n+1}_{\tau(p)}$ is exactly the map $\tilde{g}_n$ defined in the previous section.

Now, we need to make the assumption that there exists $\varepsilon >0$ such that  for all $n \geq 2$, $\varepsilon\leq p_n\leq 1$. Thus
we restrict $\tau$ to the compact set $\widetilde{X}=[\varepsilon,1]^\N \subset X$. With this restriction,
$a$ and $b$ turn out to be continuous functions on
$\widetilde{X}$. We could also consider the $\tau$-invariant compact set  $Y=\overline{\{\tau^n(\omega),\ n\in\N\}}$ for a given $\omega=(p_1,\ldots,p_n,\ldots)$, 
with $p_i\geq\varepsilon$ and consider the restriction $\tau :Y\rightarrow Y$.\\
Whenever $f:X\rightarrow \C$ is a continuous map we consider the uniform norm $\|f\|_{\infty}=\sup_{x\in X}|f(x)|$.\\
The first point is to conjugate $P$ to a monic fibered polynomial by a fibered homeomorphism
of $X\times\mathbb{C}$ of the form: $(p,z)\mapsto (p, \lambda(p) z).$ This is the object of the following Proposition

\begin{proposition}\label{conjugacy}
Assume that $p_i\in [\varepsilon,1]$, $\varepsilon>0$, then there exist a continuous function
$\lambda : X\rightarrow \R^+$
such that $P$ is conjugated by $(p,z)\mapsto (p, \lambda(p) z)$
to $\hat P(p,z)=(\tau(p),z^d+c(p))$ where $c(p)=\lambda(\tau(p))b(p)$.
\end{proposition}
\begin{proof}
Let us define $\Lambda(p,z)= (p, \lambda(p) z)$.
The conjugacy equation is given by
$
\Lambda\circ P(p,z)=\widehat P\circ\Lambda(p,z).$
Thus, by identifying the second term of this equation, we obtain for all $(p,z)\in X\times \mathbb{C}$:
$$
\lambda(\tau(p)) (a(p)z^d+b(p)) = \lambda^d(p) z^d+c(p) \, .
$$
This provides
\begin{eqnarray}
\lambda(\tau(p)) a(p) & = & \lambda(p)^d\label{lambda} \, ,\\
\lambda(\tau(p)) b(p) & =& c(p) \, .
\end{eqnarray}
To solve the first equation, we simply define $\displaystyle \lambda(p)=\prod_{i=0}^{+\infty}a(\tau^i(p))^{1/d^{i+1}}.$
This infinite product is convergent since $1 \leq a(p)=\frac{1}{p_1}\leq \frac{1}{\varepsilon}$ and then
$$
0\leq \sum_{i=0}^{+\infty}\frac{1}{d^{i+1}}\log(a(\tau^i(p))\leq \frac{1}{d-1}\log( \frac{1}{\varepsilon}) <+\infty \, .
$$
Clearly, $\lambda$ satisfies equation~(\ref{lambda}) and
$p\mapsto \lambda(p)$ is continuous on $X$ with $1\leq \lambda(p)\leq \displaystyle \Big( \frac1\varepsilon \Big)^{\frac{1}{d-1}}$.
\end{proof}

Thus after conjugacy we can assume that the fibered polynomial takes the form $\widehat  P(p,z)=(\tau(p),z^d+c(p)).$
Given a sequence of numbers $(p_i)\in [\varepsilon,1]$, by the formulas of the previous
proposition, we have
$$
\lambda(p)=\prod_{i=1}^{+\infty}\left( \frac1{p_i}\right)^{1/d^{i}}
$$
and
$$
c(p)=b(p)\lambda(\tau(p))=-\frac{1-p_1}{p_1} \prod_{i=1}^{+\infty}\left( \frac1{p_{i+1}}\right)^{1/d^{i}} \, .
$$

To avoid heavy notation, we still denote $P(p,z)$ the fibered polynomial $\widehat  P=(\tau(p),z^d+c(p))$ and $P^n(p,z)=(\tau^n(p), P^n_p(z))$.
Now we can follow \cite{ses} and verify that most of the results there are still valid in our context.

The global filled-in Julia sets are defined by
$$
 {\bf K} = \Big\{(p,z) \in X\times \mathbb{C} \mbox{ such that }
\sup_{n\in \N}| P_p^{n}(z)| <+\infty \Big\},
$$
and for all $p\in X$
$$
K_{p}= \big\{z \in \mathbb{C} \mbox{ such that } (p,z)\in K \big\} \, .
$$
One has that $P({\bf K})={\bf K}$, $P_p(K_p)=K_{\tau(p)}$, ${\bf K}$ is a compact subset of $X\times \mathbb{C}$ and
$K_p$ is  a compact subset of $\mathbb{C}$ bounded by $1+\max_{p\in X}|c(p)|$. The fibered Julia sets are the topological boundaries $J_p=\partial K_p$, and we let $
{\bf J}=\overline{\bigcup_{p\in X}\{p\}\times J_{p}}.$

We also consider the Green function of $P$ defined by
$$
G(p,z)=G_p(z):=\lim_{n\ra +\infty } \frac{1}{d^n} \log_+|P_p^n(z)| \, ,
$$
where $\log_+(p):=\max\{0,\log(p)\}.$
Then one has Proposition 2.4 in section 2.3 of \cite{ses}:

\begin{proposition} \label{green}
The map $G:X \times \mathbb{C} \rightarrow \mathbb{R}^+$ is continuous and satisfies:
\begin{enumerate}
\item $G_p$ is harmonic in $\mathbb{C}\setminus K_p$ and $ K_p$ is exactly
the set of $z\in\mathbb{C}$ such that $G_p(z)=0$~;
\item $G$ satisfies the functional equation :
$G(P(p,z))=d G(p,z)$;
\item
there exist constants $A$ and $B$ such that for all  $|z|\geq B$,
\begin{equation}
\sup_{p\in X}{\Big|G_{p}(z)-\log|z|\Big|}\leq {A\over{|z|^2}} \, .
\label{tangence-log}
\end{equation}
\end{enumerate}
\end{proposition}

\medskip

In section $2.4$ of \cite{ses}, the continuity of the $K_p$ and $J_p$ with respect to $p$ is considered, see also \cite{d}. Let us denote $\mbox{Comp}(\C)$ the set of non-empty compact subsets of $\C$ endowed with the Hausdorff distance. Then,
by Proposition 2.9 in \cite{ses}, we have the following result.
\begin{proposition}\label{continuity}
1) The map $p\in X \mapsto K_p\in \mbox{Comp}(\C)$ is upper semi-continuous.\\
2) The map $p\in X \mapsto J_p\in \mbox{Comp}(\C)$ is  lower  semi-continuous.\\
\end{proposition}

\begin{theorem}\label{quasicircle}
Let $P(p,z)=(\tau(p),z^d+c(p))$ be a fibered polynomial over $(X,\tau)$.
If the continuous map $c: X\mapsto \mathbb{C}$ satisfies, $\max |c|<(\frac{1}{2})^{\frac{d}{d-1}}$ then
there exists $k>1$ such that the Julia sets $J_p$ are $k$-quasicircles,
i.e. $J_p$ is the image of the unit circle by a $k$-quasiconformal map.
\end{theorem}

Concerning quasiconformal mapping and quasiconformal circles we refer to Ahlfors~(\cite{Ahl}) where a lot
of characterizations and properties are provided.

Here we need to adapt the results of~\cite{ses} where only the degree $d=2$ was considered.
\begin{proof}
The first step is to find an attracting domain. Let $r:=(\frac{1}{2})^{\frac{1}{d-1}}$, we claim that the
image $P_p(\mathbb{D}(0,r))$ is compactly contained in $\mathbb{D}(0,r)$. Indeed, for all $|z|\leq r$
$$
| P_p(z)| \leq |z|^d+ |c(p)|< r^d+ \left(\frac{1}{2}\right)^{\frac{d}{d-1}}
= \left(\frac{1}{2}\right)^{\frac{d}{d-1}}+\left(\frac{1}{2}\right)^{\frac{d}{d-1}}=\left(\frac{1}{2}\right)^{\frac{1}{d-1}}.
$$
Next, we will check that this property is sufficient to guarantee that the Julia set is a "uniform" quasi-circle. We follow the main lines in \cite{ses} see also \cite{cg}.

Let us denote $V_p=P_p^{-1}(\mathbb{D}(0,r))$ and  $A(p)=V_p\setminus \overline{\mathbb{D}(0,r)}$. Note that $V_p$ is a simply connected domain and thus
$A(p)$ is an annulus. Since the map $(p,z)\in X\times \mathbb{D}(0,r) \rightarrow P_p(z)$ is continuous, by compactness of $X$ one has that the modulus
of $A(p)$ is bounded from below.\\
According to lemma (5.5) in \cite{ses}, there exists
$K>0$ and, for all $p\in X$, a $K$-quasiconformal diffeormorphism $\eta_p:\C\rightarrow\C$ such that
 $\eta_p=Id$ on the complement of $V_p$, $\eta_p$ is holomorphic on $\mathbb{D}(0,r)$ and satisfies $\eta_p(c(p))=0$.

Let us define $\widetilde P_p=\eta_p\circ P_p$ and
$$
\widetilde P^n_p=\widetilde P_{\tau^{n-1}(p)}\circ\ldots\circ\widetilde P_p \, .
$$
We also consider an ellipse field $\sigma_p$ by making it circle on $\mathbb{D}(0,r)\cup (\C\setminus K_p)$ and to be invariant under
$\widetilde P_p$ on
$$
\bigcup_{n=0}^\infty (\widetilde P^n_p)^{-1}(A(\tau^n(p)).
$$
The crucial point is that the sets $(\widetilde P^n_p)^{-1}(A(\tau^n(p))$ for $n\in\N$ are disjoint. 
Thus $\sigma_p$ is well defined and there is only one distortion in the first iteration since $\widetilde P_p$ is analytic everywhere except on $A(p)$.
Let us denote $\mu_p$ the Beltrami coefficient of  $\sigma_p$, and define $\mu_p=0$ on any remaining part of $\C$. Hence
$$
\|\mu_p\|_\infty\leq k<1.
$$
Let $\psi_p$ be the solution of the associated Beltrami equation $\overline{  \psi_p}=\mu_p \psi_p$ with the normalization $\psi_p(0)=0$ and $\psi_p$ is tangent to the identity at
infinity. It is a $k$-quasiconformal map and by construction, $\psi_{\tau(p)}\circ\widetilde P_p \circ \psi_p^{-1}$  turns out to be an analytic map of degree $d$, with the critical point of order $d-1$ at 0 and also
critical value 0. Thus, $\psi_{\tau(p)}\circ\widetilde P_p \circ \psi_p^{-1}(z)=z^d$ for all $z\in \C.$ But recall that $\widetilde P_p=P_p$ on the complement of $V_p$. This yields that
$$
P_p(\psi^{-1}_p(z))=\psi^{-1}_{\tau(p)}(z^d)\, , \ \forall \, |z|\geq r_0 \, ,
$$
and that $J_p$ is the image of the unit circle by the $k$-quasiconformal map $\psi^{-1}_p$.
\end{proof}

\begin{corollary}\label{quasicircle2}
There exist $0<\rho<1$ such that whenever $p_i\in[\rho,1] $ for all $i\geq 2$, then
there exist $\kappa>1$ such that $J_p$ is a $k$-quasicircle. Actually, we can take $\rho=2(\sqrt{2}-1) \approx 0.828$.
\end{corollary}
\begin{proof}
It suffices to prove that whenever $p_i\in[\rho,1] $ with $\rho$ suitably chosen
\begin{equation} \label{cardioid}
|c(p)|=\frac{1-p_1}{p_1} \prod_{i=1}^{+\infty}\left( \frac1{p_{i+1}}\right)^{1/d^{i}}<\left(\frac{1}{2}\right)^{\frac{d}{d-1}}.
\end{equation}
Assume $p_i\in[\rho,1]$ then $1\leq \frac{1}{p_i}  \leq \frac{1}{\rho}$ and
$$
|c(p)|\leq \left(\frac{1}{\rho}-1\right) \left( \frac{1}{\rho}\right)^{\frac{1}{d-1}}\leq \left(\frac{1}{\rho}-1\right) \left( \frac{1}{\rho}\right) \, .
$$
Since $\frac{1}{4}<\left(\frac{1}{2}\right)^{\frac{d}{d-1}}$,
 a sufficient condition that implies inequality (\ref{cardioid})
is
$$
\left(\frac{1}{\rho}-1\right) \left( \frac{1}{\rho}\right)\leq \frac1 4 \, .
$$
In particular when $\rho=2(\sqrt{2}-1)$ there is equality, thus (\ref{cardioid}) holds.
Then the statement follows from Theorem \ref{quasicircle}.
\end{proof}

\begin{figure}[htb]
\begin{center}
\includegraphics[width=10cm]{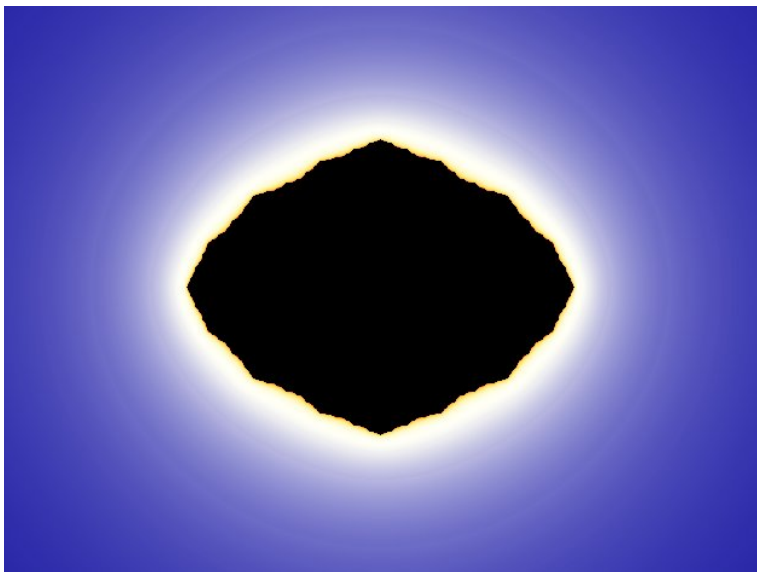}
\caption{A filled-in Julia set which is a quasi-disk. Here $(p_i)$ is a sequence of uniform random variables in $[0.83, 9]$ }
\label{Quasi-circle}
\end{center}
\end{figure}

\begin{remark}
To be more precise, associated to each $d$ there is a unique solution of the equation
$$
 \left(\frac{1}{\rho}-1\right) \left( \frac{1}{\rho}\right)^{\frac{1}{d-1}}=\left(\frac{1}{2}\right)^{\frac{d}{d-1}}
$$
which provides a value $0<\rho(d)<1$ a little more accurately.
\end{remark}
Finally, we can easily deduce the same result for the spectrum of the transition matrix $S_{d}$.
Indeed, in Theorem~\ref{spectrum} we have shown that the spectrum is exactly $E_{p}$, the set of points $z$ with bounded
orbit, i.e., such that the family ($\tilde f_{n}(z))_{n\in\N}$ is bounded. Then we conjugate $\tilde f_{n}$ to $\tilde g_{n}$ through affine maps:
$\tilde g_{n} \circ h_{1} = h_{n +1}\circ \tilde f_{n}.$
Finally, $\tilde g_{n}$ is conjugated to $P^{n}_p $ thanks to Proposition~\ref{conjugacy}. Thus we have that the following diagram is commutative:
$$\begin{CD}
\C @>\tilde f_{n}>> \C\\
@Vh_{1}VV @VVh_{n+1}V\\
\C@>\tilde g_{n}>> \C\\
@Vz\mapsto\lambda(p) z VV @VVz\mapsto\lambda(\tau^n(p)) zV\\
\C@> P^n_{p}>> \C
\end{CD}
$$
Due to the restriction $p\in [\varepsilon,1]^{\N}$, clearly
($\tilde f_{n}(z))_{n\in\N}$ is bounded if and only $(P^n_{p} (\lambda(p)h_{1}(z)))_{n\in\N}$ is bounded.
Let $\psi$ denote the inverse of $z\mapsto \lambda(p)h_{1}(z)$ which is also an affine map, it follows that
$E_{p}$ is the image of $K_{p}$ under  $\psi$. Thus the conclusions
of Proposition~\ref{continuity}, Theorem~\ref{quasicircle}, and the statement of Corollary~\ref{quasicircle2} also hold for $E_{p}$.

\section*{acknowledgments}
We are grateful to Simon Trevor Lloyd  who pointed out several mistakes in the first version and helped us to improve the readability of the final version. 
The authors would also like to thank the anonymous referee for his useful comments and remarks.

%

\end{document}